\documentclass[a4paper]{amsart}%
\usepackage{amssymb}
\usepackage{amsmath}
\usepackage{amsfonts}
\usepackage{graphicx}%
\email{alberto.marcone@dimi.uniud.it}
\email{antonio@math.uchicago.edu}
\email{shore@math.cornell.edu}
\thanks{The question answered in this paper was asked by Thierry Coquand upon attending a talk of
the first author about the results of \cite{MLE} and \cite{Mont07}.\\
Montalb\'an was partially supported by NSF grant DMS-0901169. Shore was partially supported
by NSF Grant DMS-0852811.}
\newtheorem{theorem}{Theorem}[section]

\theoremstyle{definition}
\newtheorem{definition}[theorem]{Definition}

\newtheorem{obs}[theorem]{Observation}
\newtheorem{question}[theorem]{Question}
\numberwithin{equation}{section}

\begin{document}
\title{Computing Maximal Chains}
\author{Alberto Marcone}
\address{Dipartimento di Matematica e Informatica, Universit\`{a} di Udine, 33100
Udine, Italy}
\author{Antonio Montalb\'an}
\address{Department of Mathematics, University of Chicago, Chicago, IL 60637, USA}
\author{Richard A. Shore}
\address{Department of Mathematics, Cornell University, Ithaca, NY 14853, USA}
\date{Saved: January 16, 2012.}

\subjclass[2010]{Primary: 03D80; Secondary: 06A07}

\begin{abstract}
In \cite{Wolk}, Wolk proved that every well partial order (wpo) has a maximal
chain; that is a chain of maximal order type. (Note that all chains in a wpo
are well-ordered.) We prove that such maximal chain cannot be found
computably, not even hyperarithmetically: No hyperarithmetic set can compute
maximal chains in all computable wpos. However, we prove that almost every
set, in the sense of category, can compute maximal chains in all computable wpos.

Wolk's original result actually shows that every wpo has a strongly maximal
chain, which we define below. We show that a set computes strongly maximal
chains in all computable wpo if and only if it computes all hyperarithmetic sets.
\end{abstract}

\maketitle

\section{Introduction}

In this paper we study well partial orders (from now on wpos), that is
well-founded partial orders with no infinite antichains. In \cite{Wolk}, Wolk
proved that every wpo has a maximal chain, that is a chain of maximal order
type. We are interested in two related problems. One is determining the
computational complexity of such chains and the other is the complexity of the
process that takes one from the wpo to such a chain.

If $P$ is partially ordered by ${\leq_{P}}$, $C\subseteq P$ is a chain in $P $
if the restriction of ${\leq_{P}}$ to $C$ is linear. If $P$ is a well-founded
partial order then every chain in $P$ is a well-order and we define the height
of $P$, $\operatorname{ht}(P)$, to be the supremum of all ordinals which are
order types of chains in $P$. For $x\in P$, let $\operatorname{ht}_{P}(x)$ be
the supremum of all ordinals which are order types of chains in $P_{(-\infty
,x)}=\{\,{y\in P}\mid{y<_{P}x}\,\}$. It is easy to see that $\operatorname{ht}%
(P)=\sup\{\,{\operatorname{ht}_{P}(x)+1}\mid{x\in P}\,\}$ and that
$\operatorname{ht}_{P}(x)=\sup\{\,{\operatorname{ht}_{P}(y)+1}\mid{y<_{P}%
x}\,\}$.

\begin{definition}
Let $C$ be a chain in $P$:

\begin{itemize}
\item $C$ is \emph{maximal} if it has order type $\operatorname{ht}(P)$;

\item $C$ is \emph{strongly maximal} if, for every $\alpha<\operatorname{ht}%
(P)$, there exists a (necessarily unique) $x\in C$ with $\operatorname{ht}%
_{P}(x)=\alpha$.
\end{itemize}
\end{definition}

Of course, strongly maximal chains are maximal. While maximal chains are
maximal with respect to order type, strongly maximal chains are maximal with
respect to inclusion as well (although there exist wpos with chains which are
maximal with respect to inclusion but neither strongly maximal nor maximal).

Wolk (\cite[Theorem 9]{Wolk}) actually proved the following theorem:

\begin{theorem}
\label{Wolk} Every wpo has a strongly maximal chain.
\end{theorem}

Wolk's result appears also in Harzheim's book (\cite[Theorem 8.1.7]{Harz}).
The result was extended to a wider class of well founded partial orders by
Schmidt (\cite{Schm81}) in the countable case, and by Milner and Sauer
(\cite{MS}) in general.

We can now state precisely the questions we are interested in:

\begin{question}
\label{q1} If $P$ is a computable wpo, how complicated must maximal and
strongly maximal chains in $P$ be?
\end{question}

\begin{question}
\label{q2}How complicated must any function taking the wpo $P$ to such a
maximal chain be?
\end{question}

As usual, the computability of $P$ means that both $P\subseteq\mathbb{N}$ and
${\leq_{P}}\subseteq\mathbb{N}\times\mathbb{N}$ are (Turing) computable sets.
In our answers to these questions, we will also measure complexity in terms of
Turing computability as well as the hyperarithmetic hierarchy which is built
by iterating the Turing jump (halting problem) along computable well
orderings. (Definitions and basic facts can be found, for example, in
\cite{Sacks}.)

We answer Question \ref{q1} for strongly maximal chains by showing (Theorem
\ref{smax}) that, for every hyperarithmetic set $X$, there is a computable wpo
$P$ such that any strongly maximal chain in $P$ computes $X$. Thus any set
computing strongly maximal chains in every computable wpo must lie above all
the hyperarithmetic sets. For maximal chains we show that is far from true.
Indeed, almost every set, in the sense of category, can compute maximal chains
in every computable wpo (Theorem \ref{gen}) while such \textquotedblleft
generic\textquotedblright\ sets do not compute any noncomputable
hyperarithmetic set. On the other hand, we also show (Theorem \ref{maxnot})
that the chains must be highly noncomputable in the sense that for every
hyperarithmetic set $X$ there is a computable wpo $P$ with no maximal chain
computable from $X$.

We answer Question \ref{q2} by showing that there is no computable or even
hyperarithmetical procedure for constructing even maximal chains in computable
wpos. To be more precise, any function $f(e,n)$ such that, for every
computable wpo $P$ with index $e$, the function of $n$ determined by $f $ and
$e$ ($\lambda nf(e,n)$) is (the characteristic function of) a maximal chain in
$P$ must itself compute every hyperarithmetic set $X$ (Theorem \ref{nonunif}).
Other information about this question is also provided in \S \ref{nonuniform}.

Theorem \ref{Wolk} is somewhat similar to the better known result of de Jongh
and Parikh (\cite{JP}):

\begin{theorem}
\label{dJP} Every wpo $P$ has a maximal linear extension, i.e.\ there exists a
linear extension of $P$ such that every linear extension of $P$ embeds into
it. We call such a linear extension a maximal linear extension.
\end{theorem}


In \cite{Mont07} the second author answered the analogues of Questions
\ref{q1} and \ref{q2} for maximal linear extensions. His answer for the first
question is very different than ours for maximal and strongly maximal chains
but essentially the same for the second.

\begin{theorem}
\label{Montalban} Every computable wpo has a computable maximal linear
extension, yet there is no hyperarithmetic way of computing (an index for) a
computable maximal linear extension from (an index for) the computable wpo.
\end{theorem}

These results illustrate several interesting differences between the analysis
of complexity in terms of computability strength as done here and axiomatic
strength in the sense of reverse mathematics as is done in \cite{MLE}. (See
\cite{Sim09} for basic background in reverse mathematics whose general goal is
to determine precisely which axiomatic systems are both necessary and
sufficient to prove each theorem of classical mathematics.) From the viewpoint
of reverse mathematics, all of the theorems analyzed computationally here and
in \cite{Mont07} are equivalent. Indeed, in \cite{MLE} the first and third
author showed that Theorem \ref{dJP} and Theorem \ref{Wolk} (indeed even the
version for maximal chains) for countable wpos are each equivalent (over
$\mathsf{RCA}_{0}$) to the same standard axiom system, $\mathsf{ATR}_{0}$. As
we have explained, however, the computational analysis of these three theorems
in the sense of Question \ref{q1} are quite different.

Computable partial orders all have computable maximal linear extensions
\cite{Mont07}. Computable wpos all have hyperarithmetic maximal and even
strongly maximal chains as is shown by the proof in $\mathsf{ATR}_{0}$ of
Theorem \ref{Wolk} in \cite{MLE}. However, strongly maximal chains for
computable wpos must be of arbitrarily high complexity relative to the
hyperarithmetic sets while maximal chains can be computably incomparable with
all noncomputable hyperarithmetic sets. Yet another level of computational
complexity within the theorems axiomatically equivalent to $\mathsf{ATR}_{0}$,
is provided by K\"{o}nig's duality theorem (every bipartite graph $G$ has a
matching $M$ such that there is a cover of $G$ consisting of one vertex from
each edge in $M$). (See \cite{AMS92} for definitions.) Here \cite{AMS92} and
\cite{Sim94} show that this theorem for countable graphs is equivalent to
$\mathsf{ATR}_{0}$. On the other hand, \cite[Theorem 4.12]{AMS92} shows that
there is a single computable graph $G$ such that any cover as required by the
theorem already computes every hyperarithmetic set and so this $G$ certainly
has no such hyperarithmetic cover.

Another, less natural phrasing of our theorems produces a yet different
phenomena. If one asks, for every partial order, for either a witness that it
is not a wpo or a (strongly) maximal chain then one adds on the well known
possibilities inherent in producing descending sequences in nonwellfounded
partial orderings. A more natural (or at least seemingly so) example of a
similar behavior is determinacy for open ($\Sigma_{1}^{0}$) or clopen
($\Delta_{1}^{0}$) sets. Both versions of determinacy are reverse
mathematically equivalent to $\mathsf{ATR}_{0}$ \cite[Theorem V.8.7]{Sim09}.
Computationally, the second always has hyperarithmetic solutions (strategies)
for computable games and they are cofinal in the hyperarithmetic degrees while
the former has computable instances with no hyperarithmetic solutions (again
computing a path through a nonwellfounded tree) \cite{Bla72}. Thus, we have at
least four or five different levels of computational complexity for theorems
all axiomatically equivalent to $\mathsf{ATR}_{0}$. The phenomena exhibited by
our analysis of the existence of maximal chains seems to be new.

\section{Notation, terminology and basic observations}

In this section we fix our notation about partial orders, make a simple but
crucial observation about downward closed sets in computable wpos, and recall
the notion of hyperarithmetically generic set.

If $P$ is partially ordered by ${\leq_{P}}$ and $x,y \in P$ we write $x <_{P}
y$ for $x \leq_{P} y$ and $x \neq y$, and $x \,|_{P}\, y$ for $x \nleq_{P} y
\nleq_{P} x $.

We denote by $P_{[x,y)}$ the partial order obtained by restricting ${\leq_{P}%
}$ to the set $\{\,{z \in P}\mid{x \leq_{P} z <_{P} y}\,\}$. The notations
$P_{[x,y]}$ and $P_{(x,y)}$ are defined similarly, while $P_{[x,\infty)}$ and
$P_{(-\infty,x)}$ are obtained by restricting the order relation respectively
to $\{\,{z \in P}\mid{x \leq_{P} z}\,\}$ and $\{\,{z \in P}\mid{z <_{P}
x}\,\}$. Notice that if $P$ is computable so are all these partial orders.

\begin{definition}
A set $C \subseteq P$ is cofinal in $P$ if for every $x \in P$ there exists $y
\in C$ such that $x \leq_{P} y$.
\end{definition}

\begin{definition}
A set $I \subseteq P$ is an ideal in $P$ if it is downward closed (i.e.\ $x
\in I$ and $y \leq_{P} x$ imply $y \in I$) and for every $x,y \in I$ there
exists $z \in I$ such that $x \leq_{P} z$ and $y \leq_{P} z$.
\end{definition}

\begin{definition}
Given $x_{0},...,x_{k}\in P$ we let
\[
P_{x_{0}, \dots, x_{k}} = \{\,{x\in P}\mid{x_{0} \nleq_{P} x \land\dots\land
x_{k} \nleq_{P} x}\,\}.
\]
Notice that if $P$ is computable, so is $P_{x_{0}, \dots, x_{k}}$.
\end{definition}

\begin{obs}
\label{dc} If $P$ is a computable wpo then every downward closed subset
$D\subseteq P$ is computable. In fact $P$ wpo clearly implies the existence of
a finite set of minimal elements $\{x_{0},\dots,x_{k}\}$ in $P\setminus D $
while then $D=P_{x_{0},\dots,x_{k}}$ which is computable.
\end{obs}

We will use $\alpha$-generic and hyperarithmetically generic for Cohen forcing
(i.e. conditions are finite binary strings), as defined in detail in
\cite[\S IV.3]{Sacks}.

\begin{definition}
For $\alpha<\omega_{1}^{\mathrm{CK}}$ (i.e. $\alpha$ a computable ordinal), a
set $G$ is $\alpha$-generic if the conditions which are initial segments of
$G$ suffice to decide all $\Sigma_{\alpha}$-questions. $G$ is
hyperarithmetically generic if it is $\alpha$-generic for every $\alpha
<\omega_{1}^{\mathrm{CK}}$.
\end{definition}

We associate to an infinite set a function in the usual way, described by the
next definition. We will use this function when $G$ is generic.

\begin{definition}
If $G \subseteq\mathbb{N} $ is infinite let $f_{G}: \omega\to\omega$ be
defined by letting $f_{G}(n)$ be the number of 0's between the $n$th 1 and the
$(n+1)st$ 1 in (the characteristic function of) $G$.
\end{definition}

\section{Highly noncomputable maximal and strongly maximal chains}

We will use the following result of Ash and Knight (\cite[Example 2]{AK}):

\begin{theorem}
\label{AK} Let $\alpha<\omega_{1}^{\mathrm{CK}}$. If $A$ is a $\boldsymbol{\Pi
}_{2\alpha+1}^{0}$ set then there exists a uniformly computable sequence of
linear orders $L_{n}^{A}$ such that $L_{n}^{A}\cong\omega^{\alpha} $ for all
$n\in A$ and $L_{n}^{A}\cong\omega^{\alpha+1}$ for all $n\notin A$. Indeed,
this sequence of linear orderings can be computed uniformly in indices for
$\alpha$ as a computable ordinal and $A$ as a $\boldsymbol{\Pi}_{2\alpha
+1}^{0}$ set.
\end{theorem}

Our first results concerns strongly maximal chains in computable wpos and
shows that they indeed must be of arbitrarily high complexity in the
hyperarithmetical hierarchy.

\begin{theorem}
\label{smax} Let $\alpha<\omega_{1}^{\mathrm{CK}}$. There exists a computable
wpo $P$ such that any strongly maximal chain in $P$ computes $0^{(\alpha)}$.
\end{theorem}

\begin{proof}
Let $P$ include elements $\{\,{a_{n}}\mid{n \in\mathbb{N} }\,\}$ and
$\{\,{b_{n}^{i}}\mid{n \in\mathbb{N} , i<2}\,\}$. The partial order on these
elements is given by $a_{n} <_{P} b_{n}^{i} <_{P} a_{n+1}$ and $b_{n}^{0}
\,|_{P}\, b_{n}^{1}$.

Let $A = 0^{(\alpha)}$: since $A$ is $\boldsymbol{\Sigma}^{0}_{\alpha}$\ it is
also $\boldsymbol{\Delta}^{0}_{2\alpha+1}$ and we can apply Theorem \ref{AK}
both to $A$ and to its complement $\bar{A}$. The order $P_{(b_{n}^{0},
a_{n+1})}$ consists of $L^{A}_{n}$, while $P_{(b_{n}^{1}, a_{n+1})}$ consists
of $L^{\bar{A}}_{n}$. (Therefore all elements of one chain are incomparable
with the elements of the other chain.) This completes the definition of $P$.

Notice that for every $n$ there are exactly two disjoint chains maximal with
respect to inclusion in $P_{(a_{n}, a_{n+1})}$: one of them has length
$\omega^{\alpha+1}$, whiled the other has length $\omega^{\alpha}$. Hence
$\operatorname{ht}_{P} (a_{n})= \omega^{\alpha+1} \cdot n$ for every $n$ and
$\operatorname{ht}(P) = \omega^{\alpha+2}$. Therefore there exists only one
strongly maximal chain in $P$: the one that goes through all chains of length
$\omega^{\alpha+1}$.

Thus if $C$ is a strongly maximal chain in $P$ we have $0^{(\alpha)} =
\{\,{n}\mid{b_{n}^{1} \in C}\,\}$.
\end{proof}

Our second result shows that maximal chains can also be highly noncomputable.
In contrast to Theorem \ref{smax}, however, we do not show that they must lie
arbitrarily high up in the hyperarithmetic hierarchy. Indeed, Theorem
\ref{gen} shows that this is not the case.

\begin{theorem}
\label{maxnot} Let $\alpha<\omega_{1}^{\mathrm{CK}}$. There exists a
computable wpo $P$ such that $0^{(\alpha)}$ does not compute any maximal chain
in $P$.
\end{theorem}

\begin{proof}
We can assume $\alpha$ is a successor ordinal, so that $\alpha+1 \leq2\alpha$.

Let $P$ include elements $\{\,{a_{n}}\mid{n \in\mathbb{N} }\,\}$ and
$\{\,{b_{n}^{i}}\mid{n \in\mathbb{N} , i \leq n}\,\}$. The partial order on
these elements is given by $a_{n} <_{P} b_{n}^{i} <_{P} a_{n+1}$ and
$b_{n}^{i} \,|_{P}\, b_{n}^{j}$ for $i \neq j$.

For every $i$ let $A_{i} = \{\,{n}\mid{\exists e<n\, \Phi_{e}^{0^{(\alpha)}}
(n) =i}\,\}$ which is $\boldsymbol{\Sigma}^{0}_{\alpha+1}$ and hence
$\boldsymbol{\Pi}^{0}_{2\alpha+1}$: we can thus apply Theorem \ref{AK} to
$A_{i} $. The order $P_{(b_{n}^{i}, a_{n+1})}$ consists of $L^{A_{i}}_{n}$.
(Therefore again all elements of one chain are incomparable with the elements
of the other chains.) This completes the definition of $P$.

Notice that for every $n$ there are exactly $n+1$ chains maximal with respect
to inclusion in $P_{(a_{n}, a_{n+1})}$, and these are pairwise disjoint. Since
$n$ belongs to $A_{i}$ for at most $n$ different $i$'s, at least one of these
chains has length $\omega^{\alpha+1}$, while the shorter chains have length
$\omega^{\alpha}$. Hence $\operatorname{ht}_{P} (a_{n})= \omega^{\alpha+1}
\cdot n$ for every $n$ and $\operatorname{ht}(P) = \omega^{\alpha+2}$.
Therefore every maximal chain in $P$ goes through infinitely many chains of
length $\omega^{\alpha+1}$.

If $C$ is a maximal chain in $P$ define a partial function $\psi\leq_{T} C$ by
setting
\[
\psi(n) =
\begin{cases}
i & \text{if $\exists x \in C\, b_{n}^{i} \leq_{P} x <_{P} a_{n+1}$;}\\
\uparrow & \text{otherwise.}%
\end{cases}
\]
Notice that $\psi$ is well defined because if $x, x^{\prime}\in C$ are such
that $b_{n}^{i} \leq_{P} x <_{P} a_{n+1}$ and $b_{n}^{j} \leq_{P} x^{\prime
}<_{P} a_{n+1}$ the comparability of $x$ and $x^{\prime}$ implies $i=j$.

We now show that $\psi\neq\Phi_{e}^{0^{(\alpha)}}$ for every $e$, thereby
establishing that $C \nleq_{T} 0^{(\alpha)}$. Fix $e$. There exists $n>e$ such
that $C$ intersects $P_{(a_{n}, a_{n+1})}$ in a chain of length $\omega
^{\alpha+1}$. Thus $\psi(n)$ is defined and $n \notin A_{\psi(n)}$. In
particular $\Phi_{e}^{0^{(\alpha)}} (n) \neq\psi(n)$ and thus $\psi\neq
\Phi_{e}^{0^{(\alpha)}}$.
\end{proof}

\section{Maximal chains do not code}


In this section we prove that maximal chains in wpos can be computed from
generic sets. Here is the precise statement of our result.

\begin{theorem}
\label{gen} If $P$ is a computable wpo and $G$ a hyperarithmetically generic
set then $C \leq_{T} G$ for some maximal chain $C$ in $P$. Furthermore, if $P$
has a maximal chain of length $<\omega^{\alpha+1}$, then $2 \cdot\alpha
$-genericity of $G$ suffices.
\end{theorem}

Theorem \ref{gen} is proved in several steps. First we make some observations
that allow us to restrict our attention to computable wpos $P$ such that for
some $\alpha$, $\operatorname{ht}(P) = \omega^{\alpha}$ and $P$ has a cofinal
chain of order type $\omega^{\alpha}$. Then, under these hypothesis, we first
deal with the cases $\alpha=1$ and $\alpha=2$. Eventually, generalizing the
ideas used in the simplest cases, we prove the theorem for every $\alpha$.

\subsection{Reducing to wpos with special properties}

\label{reduction} Let $P$ be a computable wpo with $\operatorname{ht}(P) =
\gamma= \omega^{\alpha_{0}}+ \dots+\omega^{\alpha_{k}}$ with $\alpha_{0}
\geq\alpha_{1} \geq\dots\geq\alpha_{k}$. By Theorem \ref{Wolk} let
$\{\,{x_{\beta}}\mid{\beta<\gamma}\,\}$ be a strongly maximal chain in $P$
with $\operatorname{ht}_{P}(x_{\beta}) = \beta$ for every $\beta<\gamma$. For
every $i \leq k$ let $\gamma_{i} = \sum_{j<i} \omega^{\alpha_{j}}$ and $a_{i}
= x_{\gamma_{i}}$, while $a_{k+1}=\infty$. Then for all $i\leq k$ we have that
$P_{[a_{i},a_{i+1})}$ is a computable wpo and $\operatorname{ht}%
(P_{[a_{i},a_{i+1})}) = \omega^{\alpha_{i}}$. Moreover if $C_{i}$ is a maximal
chain in $P_{[a_{i},a_{i+1})}$ for every $i\leq k$ then $\bigcup_{i \leq k}
C_{i}$ is a maximal chain in $P$. Obviously, if $C_{i} \leq_{T} G$ for every
$i$ then $\bigcup_{i \leq k} C_{i} \leq_{T} G$.

This shows that to prove Theorem \ref{gen} it is enough to compute from a
hyperarithmetically generic set maximal chains for computable wpos with height
of the form $\omega^{\alpha}$.\smallskip

If $P$ is such a computable wpo let $C$ be a maximal chain in $P$. Then the
set $I= \{\,{x \in P}\mid{\exists y \in C\, x \leq_{P} y}\,\}$ is downward
closed and hence, by Observation \ref{dc}, computable. Moreover
$\operatorname{ht}(I) = \omega^{\alpha}$ and $I$ has a cofinal chain of length
$\omega^{\alpha}$.

Thus, to prove Theorem \ref{gen} it suffices to compute from a
hyperarithmetically generic set maximal chains for computable wpos with height
of the form $\omega^{\alpha}$ which have cofinal chains of length
$\omega^{\alpha}$.\smallskip

We define for each $0<\alpha<\omega_{1}^{\mathrm{CK}}$ a computable operator
$\Phi_{\alpha}$ such that if $P$ is a partial order and $G$ is generic enough
we have:

\begin{itemize}
\item $\Phi_{\alpha}(P,G)$ is a chain in $P$ of order type at most
$\omega^{\alpha}$;

\item if $P$ has a cofinal chain of length $\omega^{\alpha}$, then
$\Phi_{\alpha}(P,G) $ has order type $\omega^{\alpha}$.
\end{itemize}

It is then clear that if $P$ is a wpo with $\operatorname{ht}(P)=\omega
^{\alpha}$ and a cofinal chain of length $\omega^{\alpha}$, then $\Phi
_{\alpha}(P,G)$ is a maximal chain in $P$. If moreover $P$ is computable then
$\Phi_{\alpha}(P,G)$ is $G$-computable, as desired.

The $\Phi_{\alpha}$s are defined by induction on $\alpha$.

\subsection{The case $\alpha= 1$}

For $\alpha=1$ we do not use the generic set at all, and thus we write
$\Phi_{1}(P)$. Given an enumeration $\{\,{x_{n}}\mid{n \in\mathbb{N} }\,\}$ of
$P $, define $\Phi_{1}(P)$ recursively as follows: let $x_{n} \in\Phi_{1}(P)$
if and only if for all $m<n$ with $x_{m} \in\Phi_{1}(P)$ we have $x_{m}
\leq_{P} x_{n}$. It is clear that $\Phi_{1}(P)$ is a chain of order type
$\leq\omega$ and if $P$ has a cofinal chain of length $\omega$ (so that it has
no maximal element), then $\Phi_{1}(P)$ has order type $\omega$.\smallskip

\subsection{The case $\alpha=2$}

We now consider explicitly the case $\alpha=2$, which is the blueprint for the
general case. We need to define the computable operator $\Phi_{2}$.

Using $G$, we define sequences $\langle a_{i}:i\in\mathbb{N}\rangle$,
$\langle\bar{b}_{i}:i\in\mathbb{N}\rangle$, $\langle k_{i}:i\in\mathbb{N}%
\rangle$ with $a_{i}\in P$, $\bar{b}_{i}\in P^{<\omega}$, $k_{i}\in\mathbb{N}$
and $a_{i}<_{P}a_{i+1}$ as follows. Let $k_{0}$ be the first $k$ such that
$f_{G}(k)$ is a code for a tuple $\langle a,\bar{b}\rangle$ with $a\in P$ and
$\bar{b}\in P^{<\omega}$. Let $a_{0}=a$ and $\bar{b}_{0}=\bar{b} $. Now, given
$k_{i},a_{i},\bar{b}_{i}$, let $k_{i+1}$ be the first $k>k_{i}$ such that
$f_{G}(k)$ is a code for a tuple $a\in P$, $\bar{b}\in P^{<\omega} $ and
$a_{i}<_{P}a$. (If $a_{i}$ happens to be maximal in $P$, we will wait forever
for $k_{i+1}$, i.e. the sequence is finite.) Let $a_{i+1}=a$ and $\bar
{b}_{i+1}=\bar{b}$.

For each $i$, let
\[
P_{i} = P_{\bar{b}_{i}} \cap P_{[a_{i},a_{i+1})}.
\]

Then let
\[
\Phi_{2}(P,G) = \bigcup_{i \in\mathbb{N} } \Phi_{1}(P_{i}).
\]
We claim that $\Phi_{2}$ is the computable operator we need.

First, $\Phi_{2}(P,G)$ is a chain, because each $\Phi_{1}(P_{i})$ is a chain
and if $i<j$ then every element of $P_{i}$ is below every element of $P_{j}$.
Since every $\Phi_{1}(P_{i})$ has order type at most $\omega$, the order type
of $\Phi_{2}(P,G)$ is at most $\omega^{2}$.

Second, we need to show that $\Phi_{2}(P,G)$ is computable uniformly in $P$
and $G$. Take $x\in P$. If $x <_{P} a_{0}$, then $x \notin\Phi_{2}(P,G)$.
Otherwise, we go through the definition of $a_{0}, \bar{b}_{0}, a_{1}, \bar
{b}_{1}, \dots$ until we find and $i$ such that either $x \in P_{[a_{i}%
,a_{i+1})}$ or $x \,|_{P}\, a_{i}$. By the $1$-genericity of $G$, we will
eventually find such an $i$. If $x \,|_{P}\, a_{i}$, then $x \notin\Phi
_{2}(P,G)$. If $x \in P_{[a_{i},a_{i+1})}$ then $x \in\Phi_{2}(P,G)$ if and
only if $x \in\Phi_{1}(P_{i}) $.

Third, we need to prove that if $P$ has a cofinal chain of length $\omega^{2}%
$, then $\Phi_{2}(P,G)$ has order type $\omega^{2}$. We claim that with this
hypothesis, there are infinitely many $i$'s such that $P_{i}$ has a cofinal
$\omega$-chain. The reason is that every $x \in P$ is bounded by an element of
the cofinal $\omega^{2}$-chain, and hence it is bounded by a whole $\omega
$-piece of this chain. That is, for each $x \in P$, there exists $a,
a^{\prime}, \bar{b}$ such that $x \leq_{P} a <_{P} a^{\prime}$ and $P_{\bar
{b}} \cap P_{[a,a^{\prime})}$ has a cofinal $\omega$-chain. So, by genericity,
we will be choosing such $a, a^{\prime}, \bar{b}$ infinitely often. Therefore,
for all such $i$, we have that $\Phi_{1}(P_{i})$ is an $\omega$-chain. Hence
$\Phi_{2}(P,G)$ has order type $\omega^{2}$.

\subsection{The general case}

Let $\alpha>0$ be a computable ordinal. We can fix a sequence $\{\,{\alpha
_{i}}\mid{i\in\mathbb{N} }\,\}$ with $\alpha_{i} \leq\alpha_{i+1} < \alpha$
such that $\omega^{\alpha}= \sum_{i\in\mathbb{N} } \omega^{\alpha_{i}}$ (if
$\alpha=\beta+1$ we can take $\alpha_{i}=\beta$ for every $i$, while if
$\alpha$ is limit it suffices to take an increasing cofinal sequence in
$\alpha$). Notice that $\omega^{\alpha}= \sum_{i\in A} \omega^{\alpha_{i}}$
whenever $A \subseteq\mathbb{N} $ is infinite.

We now define $\Phi_{\alpha}(P,G)$ as follows. Using $G$, define sequences
$\langle a_{i}: i \in\mathbb{N} \rangle$ and $\langle\bar{b}_{i}: i
\in\mathbb{N} \rangle$ with $a_{i} <_{P} a_{i+1}$ exactly as in the case
$\alpha=2$. We again let $P_{i} =P_{\bar{b}_{i}} \cap P_{[a_{i},a_{i+1})}$
and, using effective transfinite recursion, let
\[
\Phi_{\alpha}(P,G) = \bigcup_{i\in\mathbb{N} } \Phi_{\alpha_{i}}(G,P_{i}).
\]

We prove by transfinite induction on $\alpha$ that $\Phi_{\alpha}$ is a
computable operator with the desired properties.

The proof that, $\Phi_{\alpha}(P,G)$ is a chain and is computable uniformly in
$P$ and $G$ is exactly as in the case $\alpha=2$. Inductively it is clear that
if $P$ is a partial order and $G$ is generic enough $\Phi_{\alpha}(P,G)$ is a
chain in $P$ of order type $\leq\omega^{\alpha}$.

Now we need to prove that if $P$ has a cofinal chain of length $\omega
^{\alpha}$, then $\Phi_{\alpha}(P,G)$ has order type exactly $\omega^{\alpha}%
$. To this end we claim that in this case, there are infinitely many $i$'s
such that $P_{i}$ has a cofinal chain of length $\omega^{\alpha_{i}}$. The
reason is that every element $x \in P$ is below an element of the
$\omega^{\alpha}$-chain, and hence it is below a whole $\omega^{\alpha_{i}}%
$-piece of this chain, for all $i$. That is, for each $x \in P$ and each $i$,
there exists $a,a^{\prime},\bar{b}$ such that $x \leq_{P} a <_{P} a^{\prime}$
and $P_{\bar{b}}\cap P_{[a,a^{\prime})} $ has a cofinal $\omega^{\alpha_{i}}%
$-chain. So, by genericity, we will be choosing $a,a^{\prime},\bar{b}$ with
this property infinitely often. Therefore, by our induction hypothesis, for
each $i$ for which we make such a choice, we have that $\Phi(P_{i})$ is an
$\omega^{\alpha_{i}}$-chain. Therefore $\Phi_{\alpha}(P,G)$ has order type
$\omega^{\alpha}$.

\subsection{On the amount of genericity}

The only place where we need $G$ to meet complex dense sets is when we require
infinitely many $i$'s such that $P_{i}=P_{\bar{b}_{i}}\cap P_{[a_{i},a_{i+1}%
)}$ has a cofinal chain of length $\omega^{\alpha_{i}}$. Deciding if a wpo has
a cofinal chain of order type $\omega^{\alpha}$ is a $\Pi_{2\cdot\alpha}$ question:

\begin{itemize}
\item $P$ has a cofinal chain of order type $\omega$ iff it is an ideal and
has no maximal elements, which is the conjunction of two $\Pi^{0}_{2}$ conditions;

\item $P$ has a cofinal chain of order type $\omega^{\alpha}$ iff for all $i
\in\mathbb{N} $ and $x \in P$, there exists $a, a^{\prime}, \bar{b} \in P$
such that $x <_{P} a <_{P} a^{\prime}$ and $P_{\bar{b}} \cap P_{[a,a^{\prime
})}$ has a cofinal chain of order type $\omega^{\alpha_{i}}$.
\end{itemize}

\section{Nonuniformity \label{nonuniform}}

Our proof of Theorem \ref{gen} is nonuniform. In \S \ref{reduction} we first
need to know $\operatorname{ht}(P)$ and its Cantor normal form, then we need
to find the $a_{i}$, and eventually to compute $I$. Later in the proof, the
choice of the appropriate $\Phi_{\alpha}(P,G)$ is also nonuniform. The proofs
in $\mathsf{ATR}_{0}$ that there are hyperarithmetic maximal and strongly
maximal chains in every countable wpo in \cite{MLE} also show that there are
hyperarithmetic such chains for every computable wpo but are similarly
nonuniform (as are the ones for computable maximal linear extensions in
\cite{Mont07}). This nonuniformity cannot be avoided. We consider our results
in this paper.

If $L_{0}$ and $L_{1}$ are computable well-orders (of different length) we can
consider the wpo $L_{0}\oplus L_{1}$, the disjoint union of the two
well-orders. A maximal chain in $L_{0}\oplus L_{1}$ is included in some
$L_{i}$ for some $i<2$ and which one it is in is, of course, uniformly
computable from the maximal chain and the wpo. Then $L_{1-i}$ embeds in
$L_{i}$ and so $L_{i}$ is the longer chain. By Theorem \ref{AK}, with proper
choice of the $L_{j}$ as prescribed there, this decision can uniformly code
membership in any hyperarithmetic set. Thus we have the following
nonuniformity result:

\begin{theorem}
\label{nonunif}There is no hyperarithmetic procedure which calculates a
maximal chain in every computable wpo. In fact, any function $f(e,n)$ such
that, for every computable wpo $P$ with index $e$, $\lambda nf(e,n)$ is (the
characteristic function of) a maximal chain in $P$ must compute every
hyperarithmetic set $X$
\end{theorem}

As for Theorem \ref{gen}, if $G$ is hyperarithmetically generic and
$\alpha>\beta$ then $G^{(\beta)}$ does not compute $0^{(\alpha)}$. (Looking
toward the next theorem, one might also point out, that the ordinals
computable from such a $G$ are just the computable ordinals.) Combining this
fact with the previous arguments shows that the procedure of computing a
maximal chain in a computable wpo from a hyperarithmetically generic $G$
cannot be uniform either.

\begin{theorem}
There
is no recursive ordinal $\beta$, number $i$ and hyperarithmetically generic
$G$ such that for every index $e$ for a computable wpo $P$, $\lambda
n.\Phi_{i}^{G^{(\beta)}}(e,n)$ is (the characteristic function of) a maximal
chain in $P$.
\end{theorem}

\bibliographystyle{alpha}
\bibliography{MC}

\end{document}